\documentclass{amsart}
\usepackage[margin=1in]{geometry}
\usepackage{mathtools}
\usepackage{amsthm}
\usepackage{amssymb}   
\usepackage{url}

\theoremstyle{plain}
	\newtheorem*{thm}{Theorem}
	\newtheorem{lem}{Lemma}
	\newtheorem*{prp}{Proposition}

\title{Rings with Finitely Many Zero Divisors}
\author{Michael Kinyon}

\begin{document}
\maketitle
\begin{abstract}
	We give an elementary proof of a result which is not as well known as it should be: a ring with a specified finite number of zero divisors is finite, with a precise bound on its order.
\end{abstract}

In a typical introductory abstract algebra class, students are introduced to the concept of a zero divisor fairly soon after the definition of ring. Sometimes they are informally described as the ``worst case scenario'' among elements that are not invertible. Later, integral domains--commutative rings with unity with no zero divisors--are defined, and much of the rest of the course is devoted to integral domains and fields.

There are, however, interesting questions that are rarely raised. For example, what can we say about a ring with exactly one zero divisor? More generally, what can we say about a ring with a specified finite number of zero divisors?

Our goal in this note is to give an elementary proof of the following.

\begin{thm}
	Let $R$ be a ring with exactly $n > 0$ two-sided zero divisors. Then $R$ is finite and $|R|\leq (n+1)^2$. 
\end{thm}

Our conventions are as follows. 
Rings are not necessarily commutative nor do they necessarily have a unity. A nonzero element $a\in R$ is said to be a \emph{left zero divisor} if there exists $0\neq b\in R$ such that $ab=0$. \emph{Right zero divisors} are defined dually and \emph{two-sided zero divisors} are both left and right zero divisors. We use the exclusive convention common to introductory abstract algebra textbooks that $0$ is not a zero divisor. Readers who prefer the inclusive convention should make the necessary adjustments. 

The upper bound given in the theorem is sharp, and is achieved by, for example, the ring $\mathbb{Z}_{p^2} = \{0,1,\ldots,p^2-1\}$ where $p$ is a prime; the zero divisors are $p,2p,\ldots,(p-1)p$. 

N. Ganesan first proved the theorem in the commutative case \cite{Ganesan1964}, and then improved it by showing that a ring with exactly $m>0$ left zero divisors and $n>0$ right zero divisors has order no more than $\min\{(m+1)^2,(n+1)^2\}$ \cite{Ganesan1965}. K. Koh improved this further: a ring with exactly $n>0$ left zero divisors has order no more than $(n+1)^2$ \cite{Koh1967}; this is our Lemma \ref{Lem:koh} below. Koh also showed that equality is achieved precisely when $n+1$ is a power of a prime, that is, $|R| = p^{2j}$ for some $j\geq 1$. Independently, B. Corbas \cite{Corbas1969} and R. Raghavendran \cite{Raghavendran1969} described all the rings where equality in Koh's theorem is achieved.
 
Y. Hirano took the final step and proved the theorem as we have stated it with two-sided zero divisors \cite{Hirano1988}. Since there may be fewer two-sided zero divisors than strictly one-sided ones, Hirano's result is an improvement upon Koh's. However, Hirano's proof relies upon Koh's result, and we will make a similar reduction in our proof.

None of the proofs in the cited papers are difficult, but they do all assume some understanding of (one-sided) ideals and quotients. By contrast, the proof of the following lemma (Koh's theorem) could be presented early in a ring theory course, shortly after the definition of left zero divisor. We will need both the left- and right-sided cases in our proof of the theorem.

An instructor of an introductory ring theory course might wish to limit the theorem to the commutative case. If so, then the proof of Lemma \ref{Lem:koh} (with the one-sidedness dropped) will suffice.

\begin{lem}\label{Lem:koh}
	Let $R$ be a ring with exactly $n > 0$ left [right] zero divisors, then $|R|\leq (n+1)^2$.
\end{lem}
\begin{proof}
	Let $z_1,\ldots,z_n$ be the $n$ left zero divisors. Set $c\coloneq z_1$ and 
	choose $0\neq d\in R$ such that $cd = 0$. For each $x\in R$, $(xc)d = x(cd) = 0$, and thus $xc = 0$ or $xc = z_i$ for some $i = 1,\ldots,n$. Now let $A_0 = \{x\in R\,:\,xc=0\}$ and for each $i= 1,\ldots,n$, let $A_i = \{x\in R\,:\,xc=z_i\}$. Some $A_i$'s might be empty, but the nonempty sets in the family $\{A_0,A_1,\ldots,A_{n-1}\}$ form a partition of $R$.
	
	Now suppose $|R|\geq (n+1)^2 + 1 = n(n+1) + (n+2)$. By the pigeonhole principle, some $A_i$ has at least $n+2$ elements, say, $x_1,\ldots,x_{n+2}$. Then the $n+1$ elements $x_2 - x_1,\ldots,x_{n+2} - x_1$ are nonzero and distinct, and satisfy $(x_j - x_1)c = 0$ for each $2\leq j\leq n+2$. But this contradicts our assumption that there are only $n$ left zero divisors. Therefore $|R|\leq (n+1)^2$.
\end{proof}

While the proof of Lemma \ref{Lem:koh} gets to the desired upper bound as quickly as possible, the reader familiar with the basics of ring theory will see what is really going on with the sets $A_i$. The set $A_0$, called the \emph{left annihilator} of $a$, is a left ideal of $R$. For $i>0$, the nonempty $A_i$'s are additive cosets of $A_0$; this is what the differences $x_j-x_1$ in the proof essentially show. Thus for each $0\leq i\leq n$, $|A_i| = 0$ or $|A_i| = |A_0|$. Since $A_0$ consists of $0$ and left zero divisors, $|A_0|\leq n+1$, and therefore $|R| = \sum_{i=0}^n|A_i| \leq (n+1)^2$. In fact, this alternative argument does not really need the fact that $A_0$ is a left ideal, only that it is an additive subgroup. If a ring theory class is already familiar with abelian groups and their cosets, the instructor might wish to replace the second paragraph of the lemma's proof with some of the above considerations. This is much closer to Koh's original proof \cite{Koh1967}.

For the theorem itself, there is one additional fact we need.

\begin{lem}\label{Lem:canc}
	Let $R$ be a ring and let $0\neq a\in R$. Then $a$ is not a right [left] zero divisor if and only if $a$ is right [left] cancellable, that is, for all $x,y, xa=ya\implies x=y$ [$ax=ay\implies x=y$].
\end{lem}

\noindent Lemma \ref{Lem:canc} is an immediate consequence of the definition of right zero divisor and the right distributive law. In a typical introductory ring theory course, this is either proven immediately after the definition of zero divisor or is left as an exercise.

\begin{proof}[Proof of the Theorem]
Let $z_1,\ldots,z_n$ be the $n$ two-sided zero divisors. If every left zero divisor of $R$ is a two-sided divisor, then there are exactly $n$ left zero divisors, so Lemma \ref{Lem:koh} applies. Otherwise, $R$ has a left zero divisor $a$ that is not a right zero divisor. Choose $0\neq b\in R$ such that $ab=0$. Now for each $i$, $z_i a$ is a two-sided zero divisor. By Lemma \ref{Lem:canc}, $a$ is right cancellable, thus $z_1 a,\ldots z_n a$ are all distinct. Thus the subring $Ra$ contains all $n$ two-sided zero divisors. In addition, any right zero divisor $xa$ of $Ra$ is, in fact, a two-sided zero divisor since $(xa)(ba) = 0$. Thus $\{z_1 a,\ldots,z_n a\}$ is the set of all right zero divisors of $Ra$. Applying Lemma \ref{Lem:koh}, $|Ra|\leq (n+1)^2$. The element $a$ being right cancellable means that the mapping $R\to Ra; x\mapsto xa$ is injective, while it is also obviously surjective. Therefore $|R| = |Ra|\leq (n+1)^2$.
\end{proof}	

The rarity of finite noncommutative rings with unity \cite{oeis} leads to our final result.

\begin{prp}
	A ring with unity with exactly $1$ or exactly $2$ two-sided zero divisors is commutative.
\end{prp}
\begin{proof}
	Up to isomorphism, there is only one noncommutative ring with unity of order less than or equal to $9$, namely the ring of order $8$ consisting of $2\times 2$ upper triangular matrices with entries in the field $\mathbb{F}_2$. By direct calculation, that ring turns out to have $5$ two-sided zero divisors.
\end{proof}

\noindent\emph{Acknowledgment.} I would like to thank Des MacHale for some very useful email conversations, and I thank Gleb Pogudin for catching typos. In the interest of full disclosure, I should also mention that I was inspired to write this note by a \textsc{Quora} answer posted by computer scientist David Ash in response to a question asking if there are noncommutative rings with unity with exactly two right zero divisors \cite{Ash}. It was not clear if the asker was using the inclusive or exclusive convention, so Ash gave weak bounds both for $n=1$ ($|R|\leq 7$) and $2$ ($|R|\leq 12$) to answer the question. Ash's combinatorial approach motivated my proof of Lemma \ref{Lem:koh}.

\end{document}